\documentclass[12pt,a4paper]{article}

\usepackage{tabulary,amsmath,amssymb,amsfonts,amsthm}
\usepackage{latexsym}
\usepackage{graphicx} 
\usepackage{xcolor}
\usepackage{url,multirow,morefloats,floatflt,cancel,tfrupee}
\newtheorem{theorem}{Theorem}
\definecolor{lightblue}{rgb}{0.22,0.45,0.70}% for references
\usepackage[colorlinks=true,breaklinks=true,linkcolor=lightblue,citecolor=lightblue]{hyperref}

\numberwithin{equation}{section}
\numberwithin{figure}{section}
\numberwithin{table}{section}

\newcommand\cero{\boldsymbol{0}}
\newcommand\bC{\mathbf{C}}
\newcommand\bE{\mathbf{E}}

\newcommand\bH{\mathbf{H}}
\newcommand\bI{\mathbf{I}}
\newcommand\bL{\mathbf{L}}

\newcommand\bT{\mathbf{T}}

\newcommand\bV{\mathbf{V}}

\newcommand\rQ{\mathrm{Q}}

\newcommand\ff{\boldsymbol{f}}

\newcommand\bn{\boldsymbol{n}}
\newcommand\bu{\boldsymbol{u}}
\newcommand\bv{\boldsymbol{v}}

\newcommand\bnabla{\boldsymbol{\nabla}} 
\newcommand\bDelta{\boldsymbol{\Delta}} 
\def\rp{\mathrm{p}}
\def\rS{\mathrm{Sob}}
\def\qan{{\quad\hbox{and}\quad}}

\newcommand{\rZ}{\mathrm{Z}}
\newcommand\bW{\mathbf{W}}
\newcommand{\bxi}{\mbox{\boldmath{$\xi$}}}
\newcommand{\bchi}{\mbox{\boldmath{$\chi$}}}

\newtheorem{lemma}{Lemma}
\numberwithin{lemma}{section}
\newcommand\cS{\mathcal{S}}

\newcommand\cT{\mathcal{T}}

\newcommand\vdiv{\operatorname{div}}

\allowdisplaybreaks

\title{Analysis of a finite element method for the Stokes--Poisson--Boltzmann equations}

\author{Abeer F. AlSohaim\thanks{School of Mathematics, Monash University, 9 Rainforest Walk, 3800 VIC, Australia; and Department of Mathematics and Statistics, College of Science, IMSIU (Imam Mohammad Ibn Saud Islamic University), Riyadh, Saudi Arabia.   \protect\url{abeer.alsohaim@monash.edu}
} \and 
  Ricardo Ruiz-Baier\thanks{School of Mathematics, Monash University, 9 Rainforest Walk, 3800 VIC, Australia.
\protect\url{ricardo.ruizbaier@monash.edu}, \protect\url{http://orcid.org/0000-0003-3144-5822}}
 \and 
Segundo Villa-Fuentes\thanks{School of Mathematics, Monash University, 9 Rainforest Walk, 3800 VIC, Australia.  \protect\url{segundo.villafuentes@monash.edu},
\protect\url{http://orcid.org/0000-0002-0377-6555}}
}

\date{\today}

\usepackage[square,sort,comma,numbers]{natbib}

\begin{document}

\maketitle

\begin{abstract} We define a finite element method for the coupling of Stokes and nonlinear Poisson--Boltzmann equations. {The novelty in the formulation is that the coupling from the electric potential to the drag in the momentum balance is rewritten as a weighted advection term}. Using Banach's  {contraction}  principle, the Babu\v{s}ka--Brezzi theory, and the Minty--Browder theorem, we show that the governing equations have a unique weak solution. We also show that the discrete problem is well-posed,  establish C\'ea estimates, and derive convergence rates. We exemplify the properties of the proposed scheme via some numerical experiments showcasing convergence and applicability in the study of electro-osmotic flows in micro-channels.  \\


\end{abstract}

\section{Introduction}
\paragraph{Scope.} Electrically charged flows are of utmost importance in %computational 
chemistry applications that concern{,} for example, the design of nanopores for biomedical devices and the modelling of water purification systems. One of the simplest processes to transport fluid without mechanical stimulation is electro-osmosis. The mobility of the fluid due to thin double layers that attract and repel ions in an electrolyte, can be described by the Stokes equations  with a forcing term that depends on the charge of the electrolyte and an externally applied electric field. In a simplified scenario, the flow itself or its pressure gradients are not dominant enough to influence the distribution of the double layer electrostatic potential, and so the electric charge density can be simply related to the potential using (even simplified versions of) the Poisson--Boltzmann equation. 

Regarding unique solvability and finite element (FE) discretisation for the Poisson--Boltzmann equation we refer to the work \cite{chen07} (see also, e.g., \cite{holst12,iglesias22,kraus20}) that use, for example, a splitting between regular and singular contribution expansion and apply convex minimisation arguments, and show an $L^\infty$ bound for the solution using a cutoff-function approach. Here we take the regularised counterpart of this equation, which maintains the same type of  nonlinearity (a hyperbolic sine) but does not include the distributional Dirac forcing terms. In addition, here we  also include an advection term in the regularised potential equation, and following \cite{holst12} we restrict the functional space of the double layer potential to a bounded set (in turn, permitting us to have a bounded nonlinear operator). 

\paragraph{Outline.}  The remainder of this section presents the strong form of the coupled system. Section~\ref{sec:weak} is devoted to the weak formulation and well-posedness  using Banach fixed-point theory. In Section~\ref{sec:fem} we show the existence and uniqueness of discrete solution, and outline the a priori error analysis. Finally, in Section~\ref{sec:results} we provide numerical examples of convergence and fully developed electro-osmotic flows in  eccentric micro-tubes. 

\paragraph{The Stokes--Poisson--Boltzmann equations.} Let us consider a Lipschitz bounded domain in $\mathbb{R}^n$, $n=2,3$ with boundary $\partial\Omega = \Gamma_D \cup \Gamma_N$ split into two parts where different types of boundary conditions will be considered, and denote by $\bn$ the outward unit normal vector on the boundary.  The domain is filled with an incompressible fluid subjected to  pressure gradients and electric charges in an electrolyte. The set of equations that govern the stationary regime are written in {terms} of fluid velocity $\bu$, pressure $p$, and electrostatic double layer potential $\psi$, and read as follows 
\begin{subequations}
\label{eq:strong}
\begin{align}
-\mu \bDelta\bu + \nabla p & = \ff - \varepsilon \Delta \psi \bE & \text{in } \Omega, \label{eq:mom}\\
\vdiv \bu & = 0 & \text{in } \Omega,\label{eq:mass}\\
\kappa(\psi) +  \bu \cdot \nabla \psi - \varepsilon \Delta \psi & = g & \text{in } \Omega, \label{eq:pot}\\
\bu = \cero \quad \text{and} \quad \psi & = 0 & \text{on } \Gamma_D, \label{eq:bcD}\\
(\mu \bnabla \bu - p\bI)\bn = \cero \quad \text{and}\quad \varepsilon\nabla \psi \cdot \bn & = 0  & \text{on } \Gamma_N.\label{eq:bcN}
\end{align}\end{subequations}
Here $\bI$ is the identity tensor in $\mathbb{R}^{n\times n}$, $\mu$ is the fluid viscosity, $\ff \in \bL^2(\Omega)$ is a vector  of external body forces, $\bE \in \bL^\infty(\Omega)$ is an externally applied electric field (typically along the longitudinal direction, and assumed uniformly bounded by $\bar{E}>0$), $\varepsilon$ is the electric permittivity of the electrolyte, $g\in L^2(\Omega)$ is an external source/sink of potential, and 
$ \kappa(\psi) = k_0 \sinh(k_1 \psi)$ 
is the charge of the electrolyte as a nonlinear function of potential, where $k_0>0$ depends on the valence, the electron charge, and the bulk ion concentration, and $k_1>0$ depends additionally on the Boltzmann constant and the reference absolute temperature.  
Following \cite[Lem. 2.1]{holst12} we can assume that 
\begin{itemize}
\item the potential is uniformly bounded between the values $\alpha \leq 0 \leq \beta \in \mathbb{R}$. 
\end{itemize}
In addition, we assume $\kappa(0) = 0$, and that there exists $\bar{K},\underline{K}>0$ such that 
%adopt the following assumptions for the charge function. 
\begin{itemize}
\item 
%Lipschitz continuity: 
%there exists  such that 
$|\kappa(s_1) - \kappa(s_2)| \leq \bar{K} |s_1-s_2|$ for all $s_1,s_2 \in [\alpha, \beta]$, %and ; and 
 %monotonicity: 
 \item %there exists $\underline{K}>0$ such that 
 $|\kappa(s_1) - \kappa(s_2) | \geq \underline{K} |s_1-s_2|$ for all $s_1,s_2\in [\alpha, \beta]$. 
 \end{itemize}
 Homogeneity of \eqref{eq:bcD}-\eqref{eq:bcN} is only assumed to simplify the exposition, however the results remain valid for more general assumptions. %Note that inhomogeneous boundary conditions are used in the numerical experiments.
%
%%%%%%%%%%%%%%%%%%%%%%%%%%%%%%%%%%%%%%%%%%%%%%%%%%%%%%%%%%%%%%%%%%%%%%%%%%%%%%%%%%%%%
\section{Existence and uniqueness of weak solution}\label{sec:weak}
\paragraph{Weak formulation.}
Consider 
%The boundary conditions suggests 
the following  test and trial functional spaces for velocity, double layer potential, and pressure, respectively %\cred{[need to add bounds to $\Phi$]}
\begin{gather*}
\bV: = \{ \bv \in \bH^1(\Omega): \bv = \cero \ \text{on} \ \Gamma_D\},\qquad \Phi_0 := \{\varphi \in H^1(\Omega): \varphi = 0 \ \text{on} \ \Gamma_D\},\\ \Phi := \{\varphi \in H^1(\Omega): \varphi = 0 \ \text{on} \ \Gamma_D  \ \text{and} \ \alpha \leq \varphi \leq \beta\}, \qquad  \rQ:= L^2(\Omega).
 \end{gather*}
Then, multiplying \eqref{eq:mom}-\eqref{eq:pot} by suitable test functions, integrating by parts, and employing the boundary conditions, we are left with the following weak formulation: find $(\bu, p, \psi) \in \bV \times \rQ \times \Phi$ such that 
\begin{subequations}\label{eq:weak}
\begin{align}
a(\bu, \bv) + A^{\psi}(\bu,\bv) + b(\bv,p) & = F^\psi(\bv) & \forall \bv \in \bV,\\
b(\bu, q) & = 0 & \forall q \in \rQ,\\
(\kappa(\psi),\varphi) + c(\bu; \psi, \varphi)+ d(\psi,\varphi) & = G(\varphi) & \forall \varphi \in \Phi_0,
\end{align}
\end{subequations}
where the following bilinear and trilinear forms are used 
\begin{gather*}
a(\bu,\bv) := \mu \int_\Omega \bnabla \bu : \bnabla \bv, \quad b(\bv, q):= - \int_\Omega \vdiv\bv\,q, \\ c(\bv; \psi, \varphi): = \int_\Omega (\bv\cdot \nabla \psi)\varphi, \quad 
d(\psi,\varphi) := \varepsilon\int_\Omega \nabla \psi\cdot\nabla\varphi,
\end{gather*}
as well as (for a fixed $\hat{\psi}$) the following \emph{linear} functionals and bilinear form 
\[ F^{\hat{\psi}}(\bv):= \!\int_\Omega \{\ff + [g - \kappa(\hat{\psi})]\bE\}\cdot \bv, \ G(\varphi):= \!\int_\Omega g\,\varphi, \ A^{\hat{\psi}}(\bu,\bv):= \!\int_\Omega [\bu\cdot\nabla\hat{\psi}](\bE\cdot\bv). \]
{Note that the specific forms of $F^\psi(\bullet)$ and $A^\psi(\bullet,\bullet)$ do not have the electric permittivity $\varepsilon$, since the weak formulation we propose comes from} rewriting the last source term in the right-hand side of \eqref{eq:mom} using the potential equation \eqref{eq:pot}.  {This has the advantage that we do not need to integrate by parts the term containing the Laplacian of the potential in the momentum equation, and it gives us a more convenient structure for the fixed-point analysis below.} The bilinear and trilinear forms above are uniformly bounded{, that is, for all $\bu$, $\bv \in \bV$, $q \in \rQ$, $\psi$, $\hat{\psi} \in \Phi$ and $\varphi \in \Phi_0$, the following inequalities hold}
\begin{gather*}\vert a(\bu,\bv)\vert \leq \mu\, \Vert \bu\Vert_{1,\Omega}\Vert \bv\Vert_{1,\Omega}, \quad 
|b(\bv,q)| \leq \|\bv\|_{1,\Omega}\|q\|_{0,\Omega}, \\|d(\psi,\varphi)| \leq \varepsilon\|\psi\|_{1,\Omega}\|\varphi\|_{1,\Omega}, \quad 
|{c}(\bv;\psi,\varphi)|\leq C_{\rS}^2 \Vert\bv\Vert_{1,\Omega}\,\Vert\psi\Vert_{1,\Omega}\, \Vert\varphi\Vert_{1,\Omega}, \\ |A^{\hat{\psi}}(\bu,\bv)| \leq C_{\rS}^2 \bar{E}\Vert\hat{\psi}\Vert_{1,\Omega} \Vert \bu\Vert_{1,\Omega}\Vert \bv\Vert_{1,\Omega},
\end{gather*}
where it suffices to apply  Hölder's inequality and the estimate $\|\phi\|_{L^4(\Omega)}\leq C_{\rS} \|\phi\|_{1,\Omega}$, 
%with $C_{\rS}>0$ and 
valid for all $\phi\in H^1(\Omega)$ (cf. \cite[Th. 1.3.3]{quarteroni94}). So are the linear functionals  in $\bV'$ (under the assumption that $\hat{\psi}$ is in a bounded set), and $\Phi'$
\[ | F^{\hat{\psi}}(\bv) | \leq \bigl[\|\ff\|_{0,\Omega} + \bar{E}\|g\|_{0,\Omega} + \bar{K}\bar{E}\|\hat{\psi}\|_{1,\Omega}\bigr] \|\bv\|_{1,\Omega}, \quad 
|G(\varphi)| \leq \|g\|_{0,\Omega} \|\varphi\|_{1,\Omega}.\]
Using the Poincar\'e inequality $\|\phi\|_{1,\Omega}\leq C_{\rp} |\phi|_{1,\Omega}$, with $C_{\rp}>0$ and valid for all $\phi\in\{\varphi \in H^1(\Omega): \varphi = 0 \ \text{on} \ \Gamma_D\}$, it is not difficult to see that the bilinear forms $a(\bullet,\bullet)$ and $d(\bullet,\bullet)$ are coercive in $\bV$  and $\Phi$, respectively, % $c$ is skew-symmetric for $\bv \in \bV_0 := \mathrm{Ker}(B)$ (where $B:\bV\to\rQ'$ is the operator induced by the bilinear form $b$), and $B$ is surjective (see, e.g., \cite{ern04,temam01}): 
\begin{equation*}
|a(\bv,\bv)|\geq {\mu}[C_{\rp}^2]^{-1}\Vert \bv\Vert_{1,\Omega}^2  \quad \forall \bv \in \bV\, , \quad  
|d(\varphi,\varphi)|\geq \varepsilon[{C_{\rp}^2}]^{-1}\Vert \varphi\Vert_{1,\Omega}^2 \quad \forall \varphi \in \Phi_0.  \end{equation*}
%\begin{align*}
% |a(\bv,\bv)|&\geq \dfrac{\mu}{C_{\rp}^2}\Vert \bv\Vert_{1,\Omega}^2  \quad \forall \bv \in \bV, \qquad 
% |d(\varphi,\varphi)|\geq \dfrac{\varepsilon}{C_{\rp}^2}\Vert \varphi\Vert_{1,\Omega}^2 \quad \forall \varphi \in \Phi_0,\\
% c(\bv; \varphi,\varphi) & = 0 \quad \forall \bv \in \bV_0, \varphi \in \Phi_0,\qquad 
% \sup_{\bv \in \bV\setminus\{\cero\}}\frac{b(\bv,q)}{\|\bv\|_{1,\Omega}}  \geq \beta \, \|q\|_{0,\Omega} \quad \forall q\in \rQ.
%\end{align*}
Furthermore, $b(\bullet,\bullet)$ satisfies the following inf-sup condition (see, e.g., \cite{ern04})
\begin{equation*}
\sup_{\bv \in \bV\setminus\{\cero\}}\frac{b(\bv,q)}{\|\bv\|_{1,\Omega}}  \geq \beta \, \|q\|_{0,\Omega} \quad \forall q\in \rQ.
\end{equation*}
%%%%%%%%%%%%%%%%%%%%%%%%%%%%%%%%%%%%%%%%%%%%%%%%%%%%%%%%%%%%%%%%%%%%
\paragraph{Well-posedness analysis.} We will employ a fixed-point argument in combination with the Babu\v{s}ka--Brezzi theory \cite{boffi13} and the Minty--Browder theorem \cite{ciarlet13}. The analysis closely follows the approach in, e.g., \cite{alvarez21,caucao20} for Boussinesq-type of problems where one separates the incompressible flow equations from the nonlinear transport equation and then connects them back using Banach's fixed-point theorem. 
First, consider the following set
\begin{equation}\label{eq:def-W-psi}
 \rZ := \bigl\{\hat\psi \in \Phi:\quad  \|\hat\psi\|_{1,\Omega} \leq {\mu}[{2 C_\rp^2 C_\rS^2 \bar{E} }]^{-1} \bigr\},
\end{equation}
 and, for a fixed $\hat{\psi} \in \rZ$, the problem of finding $(\bu,p)\in \bV\times \rQ$ such that 
\begin{equation}\label{eq:Stokes}
\begin{array}{clc}
 a(\bu, \bv) + A^{\hat{\psi}} (\bu,\bv) + b(\bv,p) & = F^{\hat{\psi}}(\bv) & \forall \bv \in \bV,\\
b(\bu, q) & = 0 & \forall q \in \rQ.
\end{array}
    \end{equation}
Owing to the unique solvability of the Stokes equations with mixed boundary conditions  \cite{ern04} (in turn derived from the following properties of $[a+A^{\hat{\psi}}](\bullet,\bullet)$
\begin{align}\label{eq:a+A-prop}
\nonumber \vert a(\bu,\bv) + A^{\hat{\psi}}(\bu,\bv)\vert &\leq \bigl(\mu + {\mu}[{2 C_\rp^2}]^{-1} \bigr)\, \Vert \bu\Vert_{1,\Omega}\Vert \bv\Vert_{1,\Omega}, \\ 
\vert a(\bv,\bv) + A^{\hat{\psi}}(\bv,\bv)\vert &\geq {\mu}[{2 C_\rp^2}]^{-1}\,\Vert \bv\Vert_{1,\Omega}^2, 
\end{align}
which arise from the properties of $a(\bullet,\bullet)$ and $A^{\hat{\psi}}(\bullet,\bullet)$, the definition of $\rZ$, and  the continuity and stability of $b(\bullet,\bullet)$), we conclude that the operator
\[\cS^{\mathrm{flow}}: \Phi \to \bV \times \rQ, \quad \hat{\psi} \mapsto \cS^{\mathrm{flow}}(\hat{\psi}) = (\cS^{\mathrm{flow}}_1(\hat{\psi}),\cS^{\mathrm{flow}}_2(\hat{\psi})):= (\bu, p),\]
is well-defined. Moreover, the continuous dependence on data provided by the Babu\v{s}ka--Brezzi theory gives, in particular, that 
\begin{equation}\label{eq:stab-u}
\|\bu\|_{1,\Omega} \leq 2 C_\rp^2\mu^{-1} \bigl( \|\ff\|_{0,\Omega} + \bar{E}\| g\|_{0,\Omega} + \bar{K}\bar{E}\|\hat{\psi}\|_{1,\Omega}\bigr). 
\end{equation}
%%%%%%%%%%
On the other hand, associated with the velocity, we define the following set
\begin{equation}\label{eq:def-W-u}
 \bW := \bigl\{\hat\bu \in \bV:\quad  \|\hat\bu\|_{1,\Omega} \leq {\varepsilon}[{2 C_\rp^2 C_\rS^2}]^{-1} \bigr\}, 
\end{equation}
and consider, for a fixed $\hat{\bu}\in \bW$, the problem of finding $\psi \in \Phi$ such that 
\begin{equation}\label{eq:Boltzmann}
(\kappa(\psi),\varphi) + c(\hat{\bu}; \psi, \varphi) + d(\psi,\varphi)  = G(\varphi) \qquad  \forall \varphi \in \Phi_0.
\end{equation}
Using H\"older and Cauchy--Schwarz inequalities we proceed to verify that the operator associated with this problem is bounded 
\[|(\kappa(\psi),\varphi) + c(\hat{\bu}; \psi, \varphi) + d(\psi,\varphi)| \leq (\bar{K} + \varepsilon + C_{\rS}^2\|\hat{\bu}\|_{1,\Omega} )\,  \|\psi\|_{1,\Omega}\|\varphi\|_{1,\Omega}.\]
%\[|(\kappa(\psi),\varphi) + c(\hat{\bu}; \psi, \varphi) + d(\psi,\varphi)| \leq \max\{\bar{K},\varepsilon,\|\hat{\bu}\|_{1,\Omega}\} \|\psi\|_{1,\Omega}\|\varphi\|_{1,\Omega}.\]
Next we can use the coercivity of $d(\bullet,\bullet)$, the definition of $\bW$, and the monotonicity of   $\kappa(\bullet)$, to obtain the strong monotonicity  of the solution operator
\begin{align}\label{eq:strong-monotonicity}
%\begin{array}{ll}
\nonumber
&(\kappa(\psi_1) - \kappa(\psi_2), \psi_1-\psi_2)+ c(\hat{\bu}; \psi_1-\psi_2,\psi_1-\psi_2)+ d(\psi_1-\psi_2,\psi_1-\psi_2)\\
&\qquad
 \geq {\varepsilon}[{2 C_\rp^2}]^{-1} \|\psi_1-\psi_2\|^2_{1,\Omega} \qquad \forall\,\psi_1,\psi_2\in \Phi.
%\end{array}
\end{align}
Then{,} the Minty--Browder theorem gives that 
%existence of a unique 
$\exists !\, \psi \in \Phi$ so the map 
\[\cS^{\mathrm{elec}}: \bV \to \Phi, \quad \hat{\bu} \mapsto \cS^{\mathrm{elec}}(\hat{\bu}) = \psi,\]
is well-defined. In addition, using  \eqref{eq:strong-monotonicity}, we have that the solution satisfies 
\begin{equation}\label{eq:stab-psi}
\|\psi\|_{1,\Omega} \leq {2C_\rp^2}{\varepsilon}^{-1}\|g\|_{0,\Omega}.
\end{equation}
Then{,} we introduce an operator equivalent to the solution map of \eqref{eq:weak}: 
\begin{equation}\label{eq:def-T}
\bT: \bW \to \bW, \qquad {\hat\bu} \mapsto \bT({\hat\bu}) := \cS^{\mathrm{flow}}_1(\cS^{\mathrm{elec}}(\hat\bu)).
\end{equation}
\begin{lemma}\label{lem:ball-mapping}
Assume that $\ff \in \bL^2(\Omega)$ and $g \in L^2(\Omega)$ {satisfy}
\begin{equation}\label{eq:small-data-1}
4C_\rp^4 C_\rS^2[\mu\varepsilon]^{-1}\bigl( 1 + \bar{E} + 2\bar{K}\bar{E}C_\rp^2 {\varepsilon}^{-1}  \bigr)(\|\ff\|_{0,\Omega} + \|g\|_{0,\Omega}) \leq 1.
\end{equation}
Then, the  operator $\bT$ is well-defined and $\bT(\bW)\subseteq \bW$.
%has a unique fixed point $\bu$ in $\bW_{\bu}$. Equivalently, the coupled problem \eqref{eq:weak} has a unique solution $(\bu, p, \psi) \in \bV \times \rQ \times \Phi$ with $\bu\in\bW_{\bu}$. Moreover, there hold....
\end{lemma}
\begin{proof}
Given $\hat\bu\in\bW$,  \eqref{eq:small-data-1} implies  $\frac{4C_\rp^4 C_\rS^2 \bar{E}}{\mu \varepsilon} \|g\|_{0,\Omega}\leq 1$. Combining this with \eqref{eq:stab-psi} gives $\cS^{\mathrm{elec}}(\hat{\bu}) \in \rZ$. Using that $\cS^{\mathrm{elec}}$ and $\cS^{\mathrm{flow}}$ are well-defined, it follows that $\bT$ is well-defined. Moreover, from \eqref{eq:stab-u} and \eqref{eq:stab-psi}, we obtain
%Then, using the fact that $\cS^{\mathrm{elec}}$ and $\cS^{\mathrm{flow}}$ are well-defined, we conclude that $\bT$ is well-defined. Furthermore, using \eqref{eq:stab-u} and \eqref{eq:stab-psi}, we obtain 
\begin{align*}
\|\cS^{\mathrm{flow}}_1(\cS^{\mathrm{elec}}(\hat\bu))\|_{1,\Omega}
%&\leq \dfrac{2 C_\rp^2}{\mu} \Big( \|\ff\|_{0,\Omega} + \bar{E}\| g\|_{0,\Omega} + \bar{K}\bar{E}\|\cS^{\mathrm{elec}}(\hat\bu)\|_{1,\Omega}\Big)\\
&\leq 2 C_\rp^2\mu^{-1} \bigl( \|\ff\|_{0,\Omega} + \bar{E}\| g\|_{0,\Omega} + \bar{K}\bar{E}{2C_\rp^2}{\varepsilon}^{-1}\|g\|_{0,\Omega}\bigr).
\end{align*}
Combining this with the small data assumption \eqref{eq:small-data-1}, %and the second equation in \eqref{eq:Stokes}, 
we conclude that $\bT(\hat\bu)\in\bW$, thereby completing the proof.
\end{proof}

\begin{theorem}\label{th:main-continuous}
Assume that \eqref{eq:small-data-1} holds, and assume that
\begin{equation}\label{eq:small-data-2}
4 C_\rp^4 C_\rS^2 \bar{E}[\mu \varepsilon^2]^{-1}\,(\varepsilon + 2 C_\rp^2\bar{K} ) \,\|g\|_{0,\Omega} < 1.
\end{equation}
Then,   $\bT$ has a unique {fixed point} $\bu \in \bW$. Equivalently,   problem \eqref{eq:weak} has a unique solution $(\bu, p, \psi) \in \bV \times \rQ \times \Phi$ with $\bu \in \bW$. Moreover, 
%there exist $C_1, \,C_2>0$ such that
\begin{equation}\label{eq:stability}
\|\psi\|_{1,\Omega} \lesssim \|g\|_{0,\Omega} \qan \|\bu\|_{1,\Omega} + \|p\|_{0,\Omega} \lesssim
  \|\ff\|_{0,\Omega} + \| g\|_{0,\Omega}.
\end{equation}
\end{theorem}
\begin{proof}
Given $\hat\bu_1,  \hat\bu_2 \in \bW$, we let $\psi_1, \psi_2 \in \Phi_0$ such that $\cS^{\mathrm{elec}}(\hat\bu_1) = \psi_1$ and $\cS^{\mathrm{elec}}(\hat\bu_2) = \psi_2$. Using \eqref{eq:Boltzmann}, adding  $\pm c(\hat\bu_1, \psi_2, \varphi)$, taking $\varphi = \psi_1 - \psi_2$, and applying  \eqref{eq:strong-monotonicity} along with the boundedness of $c(\bullet;\bullet,\bullet)$, we obtain
\begin{equation}\label{eq:lips-elec}
\|\cS^{\mathrm{elec}}(\hat\bu_1) - \cS^{\mathrm{elec}}(\hat\bu_2)\|_{1,\Omega}\leq 2C_\rS^2C_\rp^2\varepsilon^{-1}\|\psi_2\|_{1,\Omega} \|\hat\bu_1 - \hat\bu_2\|_{1,\Omega}.
\end{equation}
Similarly, let $\hat\psi_1, \hat\psi_2 \in \Phi$ and $(\bu_1, p_1),  (\bu_2, p_2) \in \bV \times \rQ$, such that $\cS^{\mathrm{flow}}(\hat{\psi}_1) = (\bu_1, p_1)$ and $\cS^{\mathrm{flow}}(\hat{\psi}_2) = (\bu_2, p_2)$. Using \eqref{eq:Stokes}, adding $\pm A^{\hat{\psi}_1}(\bu_2, \bv)$, taking {$\bv = \bu_1 - \bu_2$}, and applying the coercivity of $a + A^{\hat{\psi}}(\bullet,\bullet)$
(c.f. \eqref{eq:a+A-prop}), along with the continuity of $A^{\hat{\psi}}(\bullet,\bullet)$ and the assumptions for $\kappa(\bullet)$, we have
\begin{equation}\label{eq:lips-flow}
\|\cS^{\mathrm{flow}}_1(\hat\psi_1) - \cS^{\mathrm{flow}}_1(\hat\psi_2)\|_{1,\Omega} \leq 2 C_\rp^2 \bar{E}\mu^{-1} (C_\rS^2 \|\bu_2\|_{1,\Omega} + \bar{K}) \|\hat\psi_1 - \hat\psi_2\|_{1,\Omega}.
\end{equation}
Thus, combining \eqref{eq:lips-elec} with \eqref{eq:lips-flow}, we have
\begin{align*}
&\|\bT(\hat\bu_1) - \bT(\hat\bu_2)\|_{1,\Omega} \leq  \|\cS^{\mathrm{flow}}_1(\cS^{\mathrm{elec}}(\hat\bu_1)) - \cS^{\mathrm{flow}}_1(\cS^{\mathrm{elec}}(\hat\bu_2))\|_{1,\Omega}\\
%&\quad \leq  \dfrac{2 C_\rp^2 \bar{E}}{\mu} (C_\rS^2 \|\bu_2\|_{1,\Omega} + \bar{K}) \|\cS^{\mathrm{elec}}(\hat\bu_1) - \cS^{\mathrm{elec}}(\hat\bu_2)\|_{1,\Omega}\\
&\quad \leq  2 C_\rp^2 \bar{E}\mu^{-1} (C_\rS^2 \|\bu_2\|_{1,\Omega} + \bar{K}) 2C_\rS^2C_\rp^2 \varepsilon^{-1} \|\cS^{\mathrm{elec}}(\hat\bu_2)\|_{1,\Omega} \|\hat\bu_1 - \hat\bu_2\|_{1,\Omega},
\end{align*}
then, using the fact that $\cS^{\mathrm{elec}}(\hat\bu_2)$ satisfies \eqref{eq:stab-psi} and that $\hat\bu_2$ {(and therefore, thanks to Lemma~\ref{lem:ball-mapping}, also $\bu_2$) lives in $\bW$}, we obtain
\begin{equation*}
\|\bT(\hat\bu_1) - \bT(\hat\bu_2)\|_{1,\Omega}  \leq  {8 C_\rp^6 C_\rS^2 \bar{E}}[{\mu \varepsilon^2}]^{-1} \bigl({\varepsilon}[{2 C_\rp^2}]^{-1} + \bar{K} \bigr) \,\|g\|_{0,\Omega} \|\hat\bu_1 - \hat\bu_2\|_{1,\Omega},
\end{equation*}
which together with \eqref{eq:small-data-2} and the Banach fixed-point theorem implies that $\bT$ has a unique fixed point in $\bW$.  %Equivalently,  there exists a unique $(\bu, p, \psi) \in \bW \times \rQ \times \Phi$ solution to \eqref{eq:weak}. 
Finally, the first estimate in \eqref{eq:stability} is derived similarly to \eqref{eq:stab-psi}, while the second follows directly from \cite[Th.    2.34]{ern04}.
\end{proof}

%%%%%%%%%%%%%%%%%%%%%%
\section{Finite element discretisation}\label{sec:fem}
%\paragraph{Galerkin scheme.}
Let us denote $\cT_h$ a shape-regular simplicial mesh of $\Omega$ with mesh-size $h:=\max\{h_K: K\in \cT_h\}$. Given  $k\geq1$,  the generalised Taylor--Hood FE spaces for the approximation of velocity, pressure, and  potential, are 
	\begin{align*}
		\bV_h& :=\{\bv_h\in \bC^0(\Omega): \bv_h\vert_K\in[\mathbb{P}_{k+1}(K)]^n,\ \forall K\in\cT_h,\ \bv_h = \cero \ \text{on}\ \Gamma_D\},\\
		\rQ_h& :=\{q_h\in C^0(\Omega): q_h\vert_K\in \mathbb{P}_{k}(K),\,K\in\cT_h\},\\
		\Phi_h& :=\{\varphi_h\in C^0(\Omega): \varphi_h\vert_K\in \mathbb{P}_{k+1}(K),\ \forall K\in\cT_h,\ \varphi_h = 0 \ \text{on}\ \Gamma_D\}.% \ \text{and} \ \alpha\leq \varphi_h \leq \beta\}.
		\end{align*}
%\cred{[what do we do with discrete bounds $\alpha,\beta$ in $\Phi_h$?]} 
The FE scheme then reads: find $(\bu_h,p_h,\psi_h)\in \bV_h\times \rQ_h\times \Phi_h$ such that 
\begin{subequations}
\label{eq:weak-h}
\begin{align}
a(\bu_h, \bv_h) + A^{\psi_h}(\bu_h,\bv_h) + b(\bv_h,p_h) & = F^{\psi_h}(\bv_h) & \forall \bv_h \in \bV_h,\\
b(\bu_h, q_h) & = 0 & \forall q_h \in \rQ_h,\\
(\kappa(\psi_h),\varphi_h) + c(\bu_h; \psi_h, \varphi_h) + d(\psi_h,\varphi_h) & = G(\varphi_h) & \forall \varphi_h \in \Phi_{h}.%,0}.
\end{align}
\end{subequations}
The unique solvability of the discrete problem can be shown using similar techniques as for the continuous counterpart. %Lets
%\begin{equation*}%\label{eq:def-W-psi}
% \rW_{\Phi_h} := \bigg\{\hat\psi_h \in \Phi_h:  \|\hat\psi_h\|_{1,\Omega} \leq \dfrac{\mu}{2 C_\rp^2 C_\rS^2 \bar{E} } \bigg\}, \quad  \bW_{\bV_h} := \bigg\{\hat\bu_h \in \bV_h:  \|\hat\bu_h\|_{1,\Omega} \leq \dfrac{\varepsilon}{2 C_\rp^2 C_\rS^2} \bigg\},  
%\end{equation*}
%the discrete vertions of $\rW_{\Phi}$ and $\bW$ (c.f.  \eqref{eq:def-W-psi} and \eqref{eq:def-W-u} respectively).
\begin{theorem}\label{th:main-discrete}
Assume that the {hypothesis} of Theorem \ref{th:main-continuous} holds. Then,  \eqref{eq:weak-h} has a unique solution $(\bu_h, p_h, \psi_h) \in \bV_h \times \rQ_h \times \Phi_h$. % with $\bu_h \in \bW_{\bV_h}$.
Moreover, 
%there exist $\widetilde C_1, \, \widetilde C_2>0$ such that
\begin{equation}\label{eq:stability-h}
\|\psi\|_{1,\Omega} \lesssim 
%\widetilde C_1
\|g\|_{0,\Omega} \qan \|\bu\|_{1,\Omega} + \|p\|_{0,\Omega} \lesssim %\widetilde C_2
 %( 
 \|\ff\|_{0,\Omega} + \| g\|_{0,\Omega} %)
 .
\end{equation}
\end{theorem}
%Note that, since we are now using a reduced problem in the kernel, then the we do not need to resort to, e.g., Temam's approach (see, e.g., \cite{temam01}). 

%%%%%%%%%%%%%%%%%%%%%%
%\paragraph{Stability and convergence.}
%First we state quasi-optimality. 
\begin{theorem}\label{theo:cea}
Let $(\bu,p,\psi), (\bu_h,p_h,\psi_h)$ be the continuous and discrete solutions. Then, assuming that the data satisfies
%is sufficiently smooth so that 
\begin{equation}\label{eq:small-data-3}
 \bigl( 1 +4\varepsilon^{-1} \bar{K}\bar{E}C_\rp^2   \bigr)4\varepsilon^{-1} C_\rS^2 C_\rp^2\|g\|_{0,\Omega} \leq 1,    
\end{equation}
we have the following bound, with hidden constant independent of $h$, 
%there exists $C_{\mathrm{C\acute{e}a}}>0$, independent of $h$, such that 
\[ \|\bu - \bu_h\|_{1,\Omega} + \|p-p_h\|_{0,\Omega} + \|\psi - \psi_h\|_{1,\Omega} 
\lesssim 
%\leq C_{\mathrm{C\acute{e}a}}\{ 
\mathrm{dist}(\bu,\bV_h)+\mathrm{dist}(p,\rQ_h)+\mathrm{dist}(\psi,\Phi_h)
%\}
. \]
\end{theorem}
\begin{proof}
%Let us write $\texttt{e}_{\bu} = \bu - \bu_h$, $\texttt{e}_{p} = p - p_h$ and $\texttt{e}_{\psi} = \psi - \psi_h$, and 
 For a given $(\hat{\bv}_h,\hat{q}_h,\hat{\varphi}_h) \in \bV_{h}\times Q_{h}\times\Phi_{h}$, let us  decompose errors 
 using 
 %these errors into $\texttt{e}_{\bu} = \bxi_{\bu} + \bchi_{\bu}$, $\texttt{e}_{p} = \xi_{p} + \chi_{p}$, $\texttt{e}_{\psi} = \xi_{\psi} + \chi_{\psi}$, with 
 $\bxi_{\bu} := \bu - \hat{\bv}_h$,  
 $\bchi_{\bu} := \hat{\bv}_h - \bu_h$, 
$\xi_{p} := p - \hat{q}_h$, 
$\chi_{p} := \hat{q}_h - p_h$, $\xi_{\psi} := \psi - \hat{\varphi}_h$ and   
$\chi_{\psi} := \hat{\varphi}_h - \psi_h$. 
Proceeding as in Theorem \ref{th:main-continuous}, we can obtain
\begin{align*}
\|\bchi_{\bu}\|+ \|\chi_{p}\| + \|\chi_{\psi}\| & \leq \frac{4}{\varepsilon}\bar{K}\bar{E}C_\rS^2 C_\rp^2\|\psi_h\|_{1,\Omega} \|\bu -\bu_h\|_{1,\Omega} + \varepsilon\|\xi_\psi\|_{1,\Omega} 
\\
& \ + C_\rS^2\|\psi\|_{1,\Omega} \|\bu -\bu_h\|_{1,\Omega} + (1+\mu)\|\bxi_{\bu}\|_{1,\Omega} + \|\xi_{p}\|_{0,\Omega} 
%+ \|\bxi_{\bu}\|_{1,\Omega} 
\\
& \ + C_\rS^2\|\bu_h\|_{1,\Omega}\|\xi_{\psi}\|_{1,\Omega} + C_\rS^2\bar{E}\|\psi_h\|_{1,\Omega}\|\xi_{\bu}\|_{1,\Omega}.
\end{align*}
Then, from 
%From the last estimate, together with the error decomposition, the estimates in 
\eqref{eq:stability} and  \eqref{eq:stability-h}, assumption \eqref{eq:small-data-3}, and the fact that $(\hat{\bv}_h,\hat{q}_h,\hat{\varphi}_h) \in \bV_{h}\times Q_{h}\times\Phi_{h}$ is arbitrary, the desired result is obtained.
\end{proof}
{Finally,} the following result is a direct consequence of Theorem~\ref{theo:cea} and standard interpolation properties for Taylor--Hood FEs \cite{ern04}. 
\begin{theorem}\label{theo:conv}
Let $(\bu,p,\psi), (\bu_h,p_h,\psi_h)$ be the continuous and discrete solutions assuming  $\bu \in \bV\cap \bH^{s+1}(\Omega)$, $p\in H^s(\Omega)$, and $\psi \in \Phi \cap H^{s+1}(\Omega)$, for some $s\in (1/2,k+1]$. Then, there exists $C>0$, independent of $h$, such that 
\[ \|\bu - \bu_h\|_{1,\Omega} + \|p-p_h\|_{0,\Omega} + \|\psi - \psi_h\|_{1,\Omega} \leq C\,h^s \{ \|\bu\|_{1+s,\Omega} +  |p|_{s,\Omega} +  \|\psi\|_{1+s,\Omega} \}. \]
\end{theorem}
%%%%%%%%%%%%%%%%%%%%%%
\section{Numerical results}\label{sec:results}

\begin{figure}[t!]
\begin{center}
\includegraphics[width=0.325\textwidth]{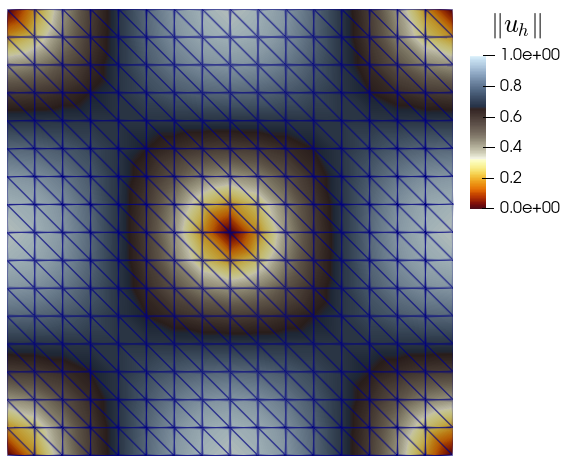}
\includegraphics[width=0.325\textwidth]{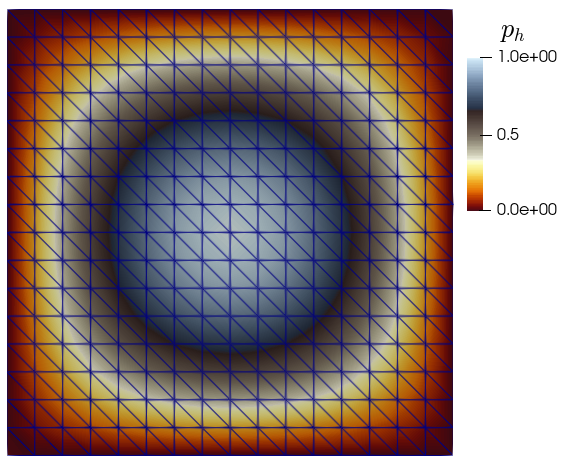}
\includegraphics[width=0.325\textwidth]{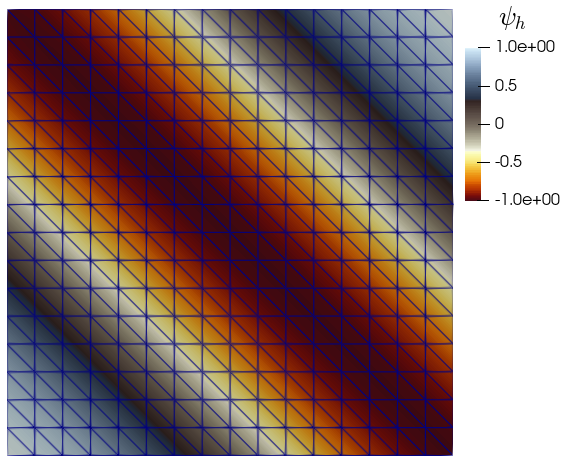}
\end{center}

\vspace{-0.4cm}
\caption{Approximate solutions 
%velocity magnitude, pressure, and potential plotted on a coarse mesh, 
for the convergence test on the unit square domain, using the lowest-order Taylor--Hood elements.}\label{fig1}
\end{figure}

We now provide {three} simple computational tests confirming the convergence of the method and  simulating electrically charged fluid in a container and a channel. They have been carried out with the  library \texttt{Gridap} \cite{badia22}. We use a  Newton--Raphson method with a residual tolerance of $10^{-7}$. 
In the first test, the convergence rates from Theorem~\ref{theo:conv} are studied using the  unit square domain $\Omega = (0,1)^2$ and  the following manufactured solutions to \eqref{eq:strong}
\begin{align*}
\bu%(x,y) 
:= \begin{pmatrix}
\cos(\pi x)\sin(\pi y)\\
-\sin(\pi x)\cos(\pi y)\end{pmatrix}, \quad   
p%(x,y) 
:=\sin(\pi x)\sin(\pi y), \quad \psi %(x,y)
:= \cos(\pi(x+y))\end{align*}
%which satisfy the incompressibility constraint. 
The Dirichlet boundary corresponds to the bottom and right segments and the Neumann boundary is the remainder. The constants assume unity values $k_0 = k_1 = \mu = \varepsilon = 1$, $\bE = (0,-1)^{\tt t}$, and the  forcing term and non-homogeneous boundary data are computed from the manufactured solutions. The error decay is reported in Table~\ref{table1}, where we also tabulate rates of convergence  $\texttt{r}(\cdot)  =\log(\texttt{e}_{(\cdot)}/\tilde{\texttt{e}}_{(\cdot)})[\log(h/\tilde{h})]^{-1}$, where $\texttt{e},\tilde{\texttt{e}}$ denote errors generated on two consecutive  meshes of sizes $h$ and~$\tilde{h}$, respectively. Optimal convergence is verified in all fields, and the approximate solutions plotted in Figure~\ref{fig1} show well-resolved profiles even for coarse meshes. 

\begin{table}[t!]
		\setlength{\tabcolsep}{3.5pt}
		\centering 
		{\footnotesize\begin{tabular}{|rcccccccc|}
\hline
				DoF   &   $h$  & $\texttt{e}(\bu)$  &   $\texttt{r}(\bu)$   &   $\texttt{e}(p)$  &   $\texttt{r}(p)$  &  $\texttt{e}(\psi)$  &   $\texttt{r}(\psi) $ &   it \\
				\hline 
				\multicolumn{9}{|c|}{$k=1$} \\
				\hline
				 57 & 0.7071 & 6.50e-01 & $\star$ & 2.21e-01 & $\star$ & 2.65e-01 & $\star$ &  4 \\
    217 & 0.3536 & 1.79e-01 & 1.860 & 3.49e-02 & 2.658 & 6.96e-02 & 1.927 &  4 \\
    849 & 0.1768 & 4.66e-02 & 1.943 & 6.97e-03 & 2.325 & 1.78e-02 & 1.969 &  4 \\
   3361 & 0.0884 & 1.18e-02 & 1.979 & 1.64e-03 & 2.089 & 4.48e-03 & 1.987 &  4 \\
  13377 & 0.0442 & 2.97e-03 & 1.992 & 4.04e-04 & 2.022 & 1.12e-03 & 1.995 &  4 \\
  53377 & 0.0221 & 7.45e-04 & 1.997 & 1.01e-04 & 2.006 & 2.82e-04 & 1.998 &  4 \\
 213249 & 0.0110 & 1.86e-04 & 1.998 & 2.51e-05 & 2.001 & 7.05e-05 & 1.999 &  4 \\
 \hline
				\multicolumn{9}{|c|}{$k=2$}\\
				\hline
133 & 0.7071 & 1.38e-01 & $\star$ & 3.01e-02 & $\star$ & 3.39e-02 & $\star$ &   4 \\
    513 & 0.3536 & 1.86e-02 & 2.899 & 4.12e-03 & 2.871 & 4.37e-03 & 2.955 &   4 \\
   2017 & 0.1768 & 2.35e-03 & 2.980 & 5.41e-04 & 2.926 & 5.51e-04 & 2.988 &   4 \\
   8001 & 0.0884 & 2.95e-04 & 2.995 & 6.87e-05 & 2.979 & 6.92e-05 & 2.995 &   4 \\
  31873 & 0.0442 & 3.69e-05 & 2.999 & 8.62e-06 & 2.995 & 8.66e-06 & 2.998 &   4 \\
 127233 & 0.0221 & 4.61e-06 & 2.999 & 1.08e-06 & 2.997 & 1.09e-06 & 2.994 &   4 \\
 508417 & 0.0110 & 5.80e-07 & 2.992 & 1.51e-07 & 2.939 & 1.63e-07 & 2.939 &   4 \\
 				\hline
		\end{tabular}}
		%\smallskip
		\caption{Convergence history. Errors, experimental rates, and Newton iteration count for FE families  using polynomial degrees $k=1,2$. }
		\label{table1}
	\end{table}

Next we study the electro-osmotic and pressure-driven mixing of a fluid in a micro-annulus following the configuration of  \cite{akyildiz19}, but instead of bipolar coordinates we employ a full 3D Cartesian system.  The model  parameters are set to $\mu = 10^{-2}$, $r_1 = 1$, $r_2 = 2$, $H = 0.75$, and $k_0 = \frac{(H-r_1-r_2)(H-r_1+r_2)(H+r_1-r_2)(H+r_1+r_2)}{2H}$, $k_1=1$, $\bE = (0,0,1)^{\tt t}$. On the outer annulus we impose $\psi = 2$, on the inner one we set $\psi=1$ and leave zero-flux boundary conditions elsewhere. We set $\bu = (0,0,1)^{\tt t}$ on the bottom face, no slip velocity on the inner and outer annulus $\bu = \cero$, and outflow conditions on the top face. For this test we discard advection (so that only the charge and inlet velocity determine the flow patterns). The results in Figure~\ref{fig2} show a higher electro-osmotic fluid mobility near the narrow part of the channel, and a distribution of the double layer potential comparable to the profiles obtained in \cite{akyildiz19}.

\begin{figure}[t!]
\begin{center}
\includegraphics[width=0.325\textwidth]{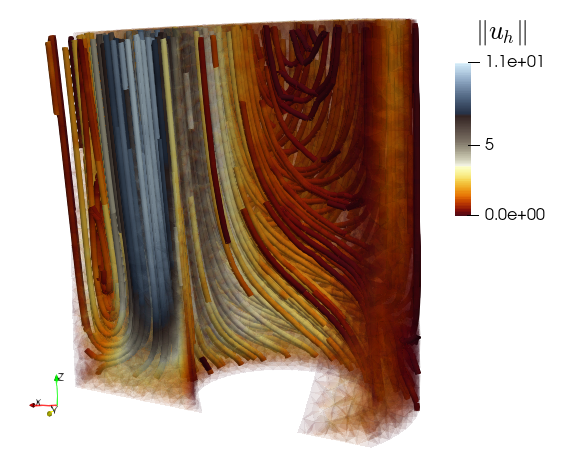}
\includegraphics[width=0.325\textwidth]{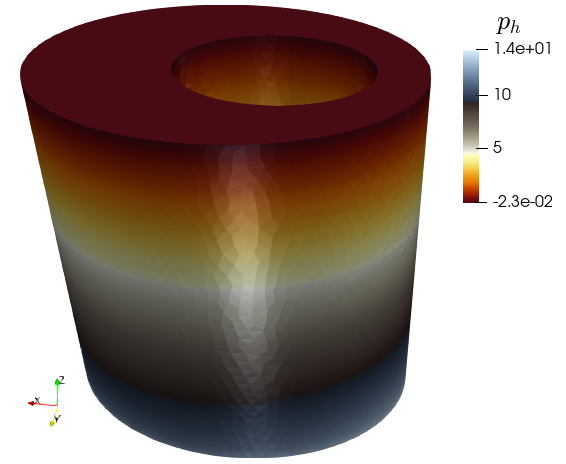}
\includegraphics[width=0.325\textwidth]{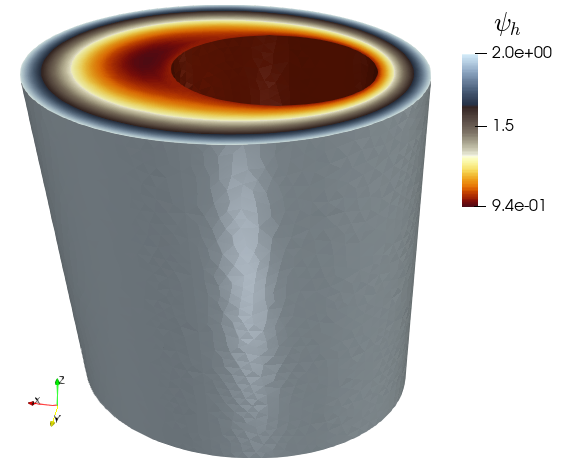}
\end{center}

\vspace{-0.4cm}
\caption{Approximate velocity magnitude and streamlines, pressure profile, and potential for the flow-potential test using Taylor--Hood elements. }\label{fig2}
\end{figure}

{Finally, we present a simulation of electrically charged flow in a nanopore sensor with obstacles. The geometry of size 12\,nm$\times$16\,nm and flow configuration are adapted from \cite[Section 4.5]{mitscha2017adaptive} (although that work focuses on Poisson--Nernst--Planck/Stokes equations). An ionic current is generated by a difference of potential on the top and bottom of the nanopore $\psi_{\max}=2$, $\psi_{\min}=0$. Also, a parabolic inflow velocity $\bu_{\mathrm{in}}= (0,-0.1\tanh(40(6-x)^2))^{\tt t}$ is imposed on the top boundary and outflow boundary conditions are considered on the bottom. On the outer left and right segments, we impose a slip condition $(\bu\cdot\bn = 0)$ and on the remainder of the boundary we impose no-slip conditions for the velocity $\bu = \cero$ and zero-flux for the potential. For this example, we incorporate again the advective term in the potential equation as well as the convective nonlinearity in the momentum balance, and consider the following parameter values (all in the nm scale)
\[ \mu = 0.1\,\mathrm{Pa}\cdot\mathrm{s}, \quad \varepsilon = 0.075, \quad \bE = (0.1,-0.1)^{\tt t}, \quad k_0 = 10^{-3}.\]
The aim of the example is simply to simulate the electric patterns and corresponding flow associated with an applied field pointing not straight down, but with a slight angle to break symmetry. In Figure~\ref{fig:ex3} we show the velocity line integral contours, the pressure drop and the potential distribution. Thanks to a comparable drag force used herein, we can see a similar recirculation as in the aforementioned reference. For this test, the total number of degrees of freedom is 157676, and the Newton--Raphson solver takes seven iterations to reach the prescribed tolerance.}

\begin{figure}[t!]
    \centering
    \includegraphics[width=0.325\linewidth]{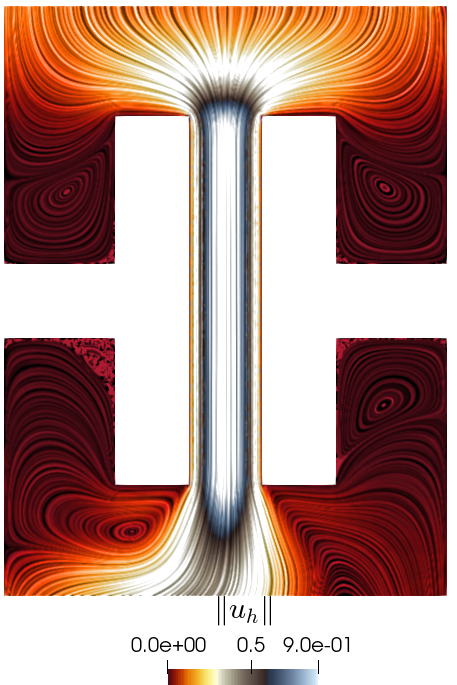}
    \includegraphics[width=0.325\linewidth]{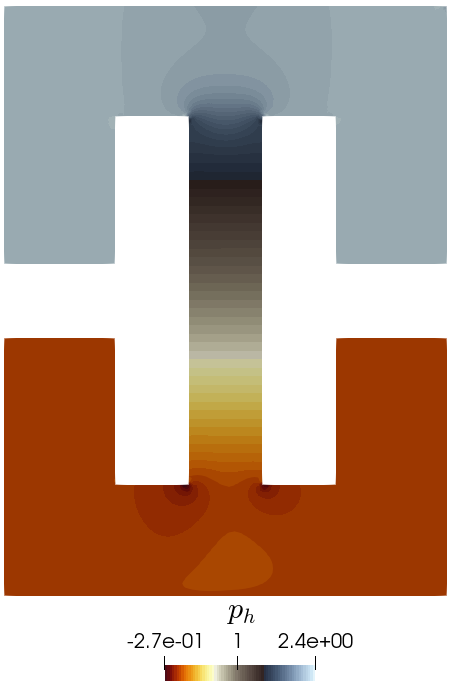}
    \includegraphics[width=0.325\linewidth]{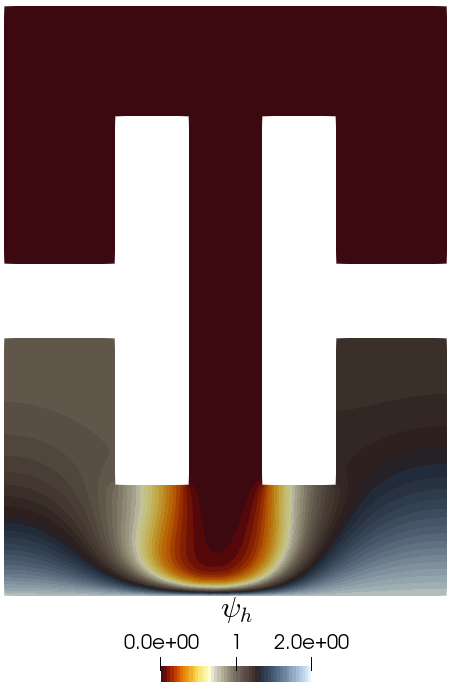}
    \caption{{Flow patterns of electrically charged fluid in a nanosensor. Approximate line integral contour of velocity, pressure, and potential.}}
    \label{fig:ex3}
\end{figure}

\small
\paragraph{Acknowledgements.} This research has been supported by the Australian Research Council through the {Future Fellowship}   FT220100496 and {Discovery Project}  DP22010316; and by the National Research and Development Agency (ANID) of  Chile through the postdoctoral grant {Becas Chile}  74220026.

\bibliographystyle{abbrvnat}
\bibliography{references}

\end{document}